%% file: Aquasiaffine.tex
\documentclass[12pt,a4paper]{amsart}
\usepackage{amsfonts}
\usepackage{amscd}
\usepackage[latin1]{inputenc}
\usepackage{t1enc}
\usepackage[mathscr]{eucal}
\usepackage{indentfirst}
\usepackage{graphicx}
\usepackage{graphics}
\usepackage{pict2e}
\usepackage{epic}
\numberwithin{equation}{section}
\usepackage[margin=2.9cm]{geometry}
\usepackage{epstopdf} 
\usepackage{amsmath}
\allowdisplaybreaks
\usepackage{amssymb}
\usepackage{amsthm}
\usepackage{bbm}
\usepackage{tikz}
\usepackage{paralist}
\usepackage[onehalfspacing]{setspace}
\usepackage{amsfonts}
\usepackage{color}
\usepackage{nicefrac}

\newcounter{dummy}
\usepackage{hyperref}
\usepackage{enumitem}
\makeatletter
\newcommand\myitem[1][]{\item[#1]\refstepcounter{dummy}\def\@currentlabel{#1}}
\makeatother
\newtheorem{thm}{Theorem}
\numberwithin{thm}{section}
\newtheorem{lem}[thm]{Lemma}
\newtheorem{prop}[thm]{Proposition}
\newtheorem{defi}[thm]{Definition}
\newtheorem{coro}[thm]{Corollary}

\newtheorem{thmintro}{Theorem}

\newtheorem*{thm*}{Theorem}
\newtheorem*{prop*}{Proposition}
\numberwithin{equation}{section}
\theoremstyle{remark}
\newtheorem{rem}[thm]{Remark}

\input{A_kuerzel}
\begin{document}
\title{A sufficient and necessary condition for $\mathcal{A}$-quasiaffinity}
\author[Schiffer]{Stefan Schiffer}
\address{Insitute for Applied Mathematics, University of Bonn} 
\email{schiffer@iam.uni-bonn.de}
\begin{abstract}
We consider a homogeneous, constant rank differential operator $\mathcal{A}$ and prove a characterisation theorem for $\mathcal{A}$-quasiaffine functions in the spirit of Ball, Currie and Olver \cite{BCO}; i.e. functions such that \[
f(v) = \int_{T_N} f(v + \psi(y))~\mathrm{d}y
\]
for all $v$ and all $\mathcal{A}$-free test functions $\psi$ with zero mean. This result is used to get a sufficient, but not necessary condition for the differential operator $\mathcal{A}$, such that linearity along the characteristic cone of $\mathcal{A}$ implies $\mathcal{A}$-quasiaffinity. We show that this implication is true if $\mathcal{A}$ admits a first order potential.
\end{abstract}
\maketitle
\input{quasiaffineintro_v2}
\input{quasiaffine1}

\input{quasiaffine2}
\input{quasiaffine3}

\bibliography{Aliterature} 
\bibliographystyle{alpha}
\end{document}

%% file: A_kuerzel.tex
\newcommand{\A}{\mathcal{A}}
\newcommand{\B}{\mathcal{B}}

\newcommand{\Bop}{\mathbb{B}}
\newcommand{\V}{\vert}
\newcommand{\norm}{\Vert}
\newcommand{\R}{\mathbb{R}}
\newcommand{\N}{\mathbb{N}}
\newcommand{\Z}{\mathbb{Z}}

\newcommand{\D}{\mathcal{D}}

\newcommand{\Sphere}{\mathbb{S}}
\newcommand{\back}{\backslash}

\newcommand{\dx}{\,\textup{d}x}
\newcommand{\dt}{\,\textup{d}t}
\newcommand{\dy}{\,\textup{d}y}

\renewcommand{\phi}{\varphi}

\newcommand{\av}{\#}

\everymath{\displaystyle}
\newcommand{\weakto}{\rightharpoonup}
\newcommand{\weakstar}{\overset{\ast}{\rightharpoonup}}

\DeclareMathOperator{\Lin}{Lin}

\DeclareMathOperator{\Image}{Im}
\DeclareMathOperator{\curl}{curl}

\DeclareMathOperator{\spann}{span}
\DeclareMathOperator{\divergence}{div}

\DeclareMathOperator{\sgn}{sgn}

\DeclareMathOperator{\loc}{loc}

%% file: quasiaffineintro_v2.tex
\section{Introduction} \label{sec:intro}
\subsection{$\mathcal{A}$-quasiaffine functions}
Let us consider a homogeneous differential operator $\A \colon C^{\infty}(\R^N,\R^d) \to C^{\infty}(\R^N,\R^l)$ of order $k_{\A}$, given by \begin{equation} \label{intro:Adef}
\A u = \sum_{\V \alpha \V = k_{\alpha}} A_{\alpha} \partial_{\alpha} u,\end{equation}
where $A_{\alpha} \in \Lin(\R^d,\R^l)$ are constant coefficients. We call $f \colon \R^d \to \R$ \textbf{$\A$-quasiaffine} if for all $\A$-free test functions on the $N$-torus, i.e. $\psi \in C^{\infty}(T_N,\R^d)$ with $\A \psi=0$ and $\int_{T_N} \psi=0$, and all $v \in \R^d$ \begin{equation} \label{intro:AqAdef}
f(v) = \int_{T_N} f(v + \psi(y)) \dy.
\end{equation}
In this work, we prove a sufficient and necessary condition for a function $f$ to be  $\A$-quasiaffine depending on derivatives of $f$.

Let us shortly recall, that for differential operators as in \eqref{intro:Adef} and $\xi \in \R^N \back \{0\}$, the Fourier symbol $\A[\xi]$ of $\A$ is defined as \[
\A [\xi] = \sum_{\V \alpha \V =k_{\alpha}} A_{\alpha} \xi^{\alpha} \in \Lin(\R^d,\R^l)
\]
We assume that $\A$ satisfies the following two conditions: \begin{enumerate}[label=(\alph*)]
    \item \textbf{Constant rank property:} The rank of the linear operator $\A[\xi]$ is constant, i.e. there is $r \in \N$, such that  for all $\xi \in \R^N \back \{0\}$ \[
    \dim \ker \A[\xi] = r; \]
    \item \textbf{Spanning property: } \[ \mathrm{span} \left\{\bigcup_{\xi \in \R^N \back \{0\}} \ker \A[\xi] \right\} =\R^d. \]
\end{enumerate}
Recently, \textsc{Rai\c{t}\u{a}} proved another characterisation of constant rank operators \cite{Raita}. Namely $\A$ has constant rank if and only if it admits a \textit{potential} $\B \colon C^{\infty}(\R^N,\R^m) \to C^{\infty}(\R^N,\R^d)$, i.e. a differential operator $\B$, such that its Fourier symbol $\B[\xi]$ satisfies \[
\Image \B [\xi]= \ker \A [\xi]
\]
for all $\xi \in \R^N \back \{0\}$. Using this result, we are able to give several equivalent conditions of what it means to be an $\A$-quasiaffine function.
\begin{thmintro} \label{thm:A}
Let $f \colon \R^d \to \R$ and let $\A$ satisfy the constant rank property and the spanning property. Moreover, let $\B$ be a potential of $\A$. Then the following statements are equivalent. \begin{enumerate}[label=(\alph*)]
\item \label{thmA:1} $f$ is $\A$-quasiaffine;
\item \label{thmA:2}$f$ is a polynomial and $\forall x \in \R^d$, $\forall  r \geq 2$, $\forall \xi_1,...,\xi_r \in \R^d$ which are linear dependent and $ \forall v_1,...,v_r \in \R^d$ with $v_i \in \ker \A[\xi_i]$ we have \begin{equation} \label{Aqax}
D^r f (x) [v_1,...,v_r] =0;
\end{equation}
\end{enumerate}
Moreover, we have the following equivalent conditions (cf. \cite{RG})
\begin{enumerate}[label=(\alph*)] \setcounter{enumi}{2}
\item \label{thmA:3}$f$ is $C^1$ and the Euler-Lagrange equation \begin{equation} \label{EL:intro}
\B ^T (\nabla f (\B u)) = 0
\end{equation}
is satisfied in the sense of distributions $\forall u \in C^{k_{\B}}(\bar{\Omega})$, i.e. for all $\phi \in C_c^{\infty}(\Omega,\R^m)$ we have \[
\int_{\Omega} \nabla f (\B u) \cdot \B \phi =0;
\]
\item \label{thmA:4}The map $u \to f(u)$ is sequentially weak$*$ continuous from $L^{\infty}(\Omega,\R^d) \cap \ker \A$ to $L^{\infty}(\Omega,\R^d)$, i.e. if $u_n \in L^{\infty}(\Omega,\R^d)$ with $\A u_n=0 $ and $u_n \weakstar u$ in $L^{\infty}(\Omega,\R^d)$, then also $f(u_n) \weakstar f(u)$ in $L^{\infty}(\Omega,\R^d)$;
\item \label{thmA:5}$f$ is a polynomial of degree $s \leq d$ and the map $u \to f(u)$ is sequentially weakly continuous from $L^s(\Omega,\R^d)$ to $\D'(\Omega)$ (the space of distributions on $\Omega$), i.e. if $u_n \in L^{s}(\Omega,\R^d)$ with $\A u_n=0 $ and $u_n \weakto u$ in $L^{s}(\Omega,\R^d)$, then  \[
\lim_{n \to \infty} \int_{\Omega} f(u_n) \phi= \int_{\Omega} f(u) \phi \quad \forall \phi \in C_c^{\infty}(\Omega).
\]
\end{enumerate}
\end{thmintro}

We prove that condition \ref{thmA:2} can in fact be weakened, indeed we may only consider $ 2 \leq r \leq \min \{ N,k_{\B} \}+1$ instead of $r \geq 2$ (cf. Theorem \ref{main:2}). This can be used to show the following result.

\begin{thmintro} \label{thm:B}
Let $\A$ be a constant rank operator satisfying the spanning property and let $\B$ be a potential of $\A$. Suppose that $\B$ is of order one. Then $f \colon \R^d \to \R$ is $\A$-quasiaffine if and only if $f$ is $\Lambda_{\A}$-affine, i.e. for all $v \in \R^d$ and $w \in \Lambda_{\A} = \cup_{\xi \in \R^N \back \{0\}} \ker \A[\xi]$ we have that \[
 t \longmapsto f(v +tw) \quad \text{is affine.}
\]
\end{thmintro}
\subsection{Quasiaffinity and Null-Lagrangians}
Let us give a few historical remarks, before finishing the introduction with a few applications of $\A$-quasiaffinity.
The study of $\A$-quasiaffine functions started with the operator $\A = \curl$, i.e. with the potential operator $\B =\nabla \colon C^{\infty}(\R^N,\R^m) \to C^{\infty}(\R^N,\R^{N \times m})$. These $\curl$-quasiaffine functions are also often called \textit{Null-Lagrangians}. In this setting, it is a well-known result, that all $\curl$-quasiaffine functions are linear combinations of minors of $N \times m$-matrices (e.g. \cite{Mor66,Resh,Conti,Dac}). \textsc{Ball}, \textsc{Currie} and \textsc{Olver} considered the potential operator of higher gradients $\B = \nabla^k$ and showed an analogue of Theorem \ref{thm:A} in this special setting \cite{BCO}. In particular, their work provides us with an example showing that the statement of Theorem \ref{thm:B} fails in the case, where the potential $\B$ has degree larger than one. We shall mention again that the equivalences \ref{thmA:3}- \ref{thmA:5} in Theorem \ref{thm:A} are treated in the general setting of constant rank operators $\A$ by \textsc{Guerra} and \textsc{Rai\c{t}\u{a}} in \cite{RG}. In particular \ref{thmA:4} and \ref{thmA:5} are examined in great detail.

\subsection{$\A$-quasiaffine functions and minimisation problems}
Let us outline a set of problems, where $\A$-quasiaffine functions play an important role.  Consider the functional $I \colon L^p(\Omega,\R^d) \to [0,\infty)$,
\begin{equation} \label{intro:func}
I(u) = \left\{ \begin{array}{ll}
     \int_{\Omega} f( u(y)) \dy &\text{if }\A u=0, \\
     \infty & \text{else,}  
\end{array} \right.
\end{equation}
A powerful tool to show that $I$ posseses a minimiser is the \textit{Direct Method}. To get this method to work we need to show the following: \begin{enumerate}[label=(\arabic*)]

    \item $I$ is bounded from below and there exists $u$, such that $I(u) < \infty$;
    \item \label{prop2}$I$ is weakly lower semi-continuous, i.e. if $u_n \weakto u$, then $I(u) \leq \liminf_{n \to \infty} I(u_n)$;
    \item \label{prop3}$I$ is \textit{coercive}, i.e. for every $C>0$ there is $R>0$, such that $\norm u \norm_{L^p} \geq C.$
\end{enumerate}
$\A$-quasiaffinity comes into play for \ref{prop2} and \ref{prop3}. On the one hand, we note that if $f$ is $\A$-quasiaffine, the functional $I$ in \eqref{intro:func} is already weakly continuous. Weak lower-semicontinuity of $I$ is equivalent to the notion of $A$\textit{-quasiconvexity} \cite{FM}, where we relax \eqref{intro:AqAdef} to \[
f(v) \leq \int_{T_N} f(v + \psi(y)) \dy
\]
for all suitable test functions. It is, however, not so easy to show that a given function $f$ is $\A$-quasiconvex. Therefore, one often studies functions of the form $g(f(v))$, where $f \colon \R^d \to \R^e$ is a component-wise $\A$-quasiaffine function and $g \colon \R^e \to \R$ is convex, which is, in the setting $\A= \curl$, referred to as \textit{polyconvexity} (e.g. \cite{Alibert,Ball77}). A short calculation shows that any function of the form $g(f(\cdot))$ for $\A$-quasiaffine $f$ and convex $g$ is already $\A$-quasiconvex.

On the other hand, $\A$-quasiaffine functions also can be used to consider non-standard coercivity conditions. Usually, we may just assume that $f(v) \geq C_1 \V v \V^p-C_2$. Given some further restrictions on the problem, we may use a coercivity condition using $\A$-quasiaffine functions. As an example, consider $\Omega \subset \R^N$ open and bounded, $J \colon W^{1,N}(\Omega,\R^N)$, $u_0 \in W^{1,N}(\Omega,\R^N)$, $f \colon \R^{N \times N} \to [0,\infty)$ and
\begin{equation} \label{intro:func2}
J(u) = \left\{ \begin{array}{ll}
     \int_{\Omega} f( \nabla u(y)) \dy &\text{if } u - u_0 \in W^{1,N}_0(\Omega,\R^N)  \\
     \infty & \text{else,}  
\end{array} \right.
\end{equation}
Then the (non-standard) coercivity condition including the quasiaffine function $\det$ \[
f(A) \geq C \V A \V^N - C_2 - \textup{det}(A) \quad \text{for all } A \in \R^{N \times N}
\]
guarantees, that minimising sequences of $I$ are bounded in  $W^{1,N}(\Omega,\R^d)$. Moreover, if $I$ is weakly lower-semicontinuous, this growth condition then implies existence of minimisers.
\subsection{Compensated Compactness}
As one might expect in view of Theorem \ref{thm:A} \ref{thmA:4} and \ref{thmA:5}, the notion of $\A$-quasiaffinity plays a crucial role in the theory of compensated compactness (e.g. \cite{Murat2,DiPerna,Rindler,RG}).
In particular, the classical $\divergence$-$\curl$ Lemma saying that if $u_n,v_n \in L^2(\Omega,\R^d)$ satisfy $\divergence u_n=0$ and $\curl v_n=0$, then\[
u_n \weakto u,~v_n \weakto v \text{ in } L^2 \Longrightarrow u_n \cdot v_n \weakto u \cdot v \text{ in the sense of distributions}
\]
can be seen from the fact that $u \cdot v$ is a $\divergence$-$\curl$ quasiaffine function (cf. Section \ref{sec:divcurl}, \cite{Murat2,Compcomp,Murat}) .

\subsection{Outline} This paper is organised as follows. In Section \ref{sec:conv}, we recall basic facts about constant rank operators, their potentials and $\A$-quasiconvexity. We deal with $\A$-quasiaffine functions and prove Theorem \ref{thm:A} and Theorem \ref{thm:B} in Section \ref{sec:affin}. Section \ref{sec:examples} presents a short coverage of examples of $\A$-quasiaffine functions for some well-known operators.

%% file: quasiaffine1.tex
\section{$\A$-quasiconvex functions} \label{sec:conv}
\subsection{Notation}
Denote by $e_1,...,e_N$ the standard basis of $\R^N$.
For $N \in \N$ let us define the $N$-torus as $[0,1]^N$ with the usual identifications of faces. We may identify functions $u \in W^{k,p}(T_N,\R^d)$ with $\Z^N$-periodic functions in $W^{k,p}_{\loc}(\R^N,\R^d)$.
For $u \in W^{1,p}(\Omega,\R^d)$, we write \[
\partial_j u = \frac{\partial}{\partial x_j} u
\]
to denote partial derivatives. With $\nabla u \in L^p(\Omega,\R^{N \times d})$ we denote the gradient of $u$, a matrix consisting of the entries $\partial_j u_i$ and, likewise, $\nabla^r$ is the $r$-th gradient. In contrast to this, for $f \colon \R^d \to \R$,  we denote by $D^r f$ the $r$-th derivative seen as a multilinear map from $(\R^d)^r$ to $\R$.

\noindent For a multiindex $\alpha=(\alpha_1,...,\alpha_N)$ and for $\lambda=\sum_{i=1}^N \lambda_i e_i \in \R^N$ write  \[
\V \alpha \V := \sum_{i=1}^N \alpha_i,\quad \lambda^{\alpha}= \prod_{i=1}^N \lambda_i^{\alpha_i}.
\]
For $u \in C^{\V \alpha \V}(\R^N,\R^d)$ we write \[
\partial_{\alpha} u = \partial_{1}^{\alpha_1}...\partial_{n}^{\alpha_n} u.
\]
 For a function space $X \subset L^1(T_N,\R^d)$, we define  \[
 X_{\av} = \{ u \in X \colon \int_{T_N} u =0 \} .\]

\subsection{Constant Rank Operators}
Consider a differential operator $\A \colon C^{\infty}(\R^N,\R^d) \to C^{\infty}(\R^N,\R^l)$ with constant coefficients  given by \begin{equation}
    \label{A1}
    \A u = \sum_{\V \alpha \V =k_{\A}} A_{\alpha} \partial_{\alpha} u,
\end{equation}
where $A_{\alpha} \in \Lin (\R^d,\R^l)$. We denote by $\A^{\ast}$ the adjoint operator of $\A$, \\i.e. $\A^{\ast} \colon C^{\infty}(\R^N,\R^l) \to C^{\infty}(\R^N,\R^d)$ and 
\[
\A^{\ast} v = (-1)^{k_{\A}}\sum_{\V \alpha \V =k_{\A}} A_{\alpha}^T \partial_{\alpha} u.
\]
We say that $u \in L^p(\R^N,\R^d) \cap \ker \A$, if $u \in L^p(\R_N,\R^d)$ and for all $v \in C^{\infty}_c(\R^N,\R^d)$ \[
\int_{\R^N} u \cdot \A^{\ast} v =0.
\]
Likewise, we can also define what it means for $u \in L^p_{\loc}(\R^N,\R^d)$ to be in $\ker \A$ and hence also for $u \in L^p(T_N,\R^d)$. Let us recall some basic notions for the operator $\A$ (c.f. \cite{Murat2,Murat,Compcomp}).

\begin{defi} Let $\A$ be a differential operator as in \eqref{A1}. \begin{enumerate}[label=(\alph*)]
\item The Fourier symbol of the differential operator $\A$ is the map $\R^N \back \{0\} \to \Lin(\R^d,\R^l)$ defined by \[
\A [\xi] = \sum_{\V \alpha \V=k_{\A}} \xi^{\alpha}A_{\alpha};
\] 
\item $\A$ satisfies the \textbf{constant rank property}  if there exists an $r \in \{0,...,N\}$ such that \[
 \dim \ker \A [\xi] = r;
\]
\item The \textbf{characteristic cone} $\Lambda_{\A} \subset \R^d$ of $\A$ is defined by \[
\Lambda_{\A} := \bigcup_{\xi \in \R^N \back \{0\}} \ker \A[\xi].
\]
\item $\A$ satisfies the \textbf{spanning property}, if the characteristic cone $\Lambda_{A}$ spans up $\R^d$.

\end{enumerate}
\end{defi}
In adittion, we also consider a differential operator $\B \colon C^{\infty}(\R^N,\R^m) \to C^{\infty}(\R^N,\R^d)$ of order $k_{\B}$, given by \[
\B u = \sum_{\V \alpha \V = k_{\B}} B_{\alpha} \partial_{\alpha} u.
\]
Likewise, we define $\B^{\ast}$ and the cone $\Lambda_{\B}$ for $\B$.

\begin{defi} We call a differential operator $\B$ the potential of $\A$ if $\forall \xi \in \R^N \back \{0\}$ we have $\Image \B [\xi] = \ker \A [\xi]$.
\end{defi}

\begin{rem} As it was pointed out in \cite{RG}, the potential $\B$ is not unique, even if we fix the order $k_{\B}$ of $\B$ and identify operators via homeomorphisms of the underlying space $\R^m$. Moreover, note that if $\B \colon C^{\infty}(\R^N,\R^m) \to C^{\infty}(\R^N,\R^d)$ is a potential of $\A$, then also \[
\B \circ \divergence \colon C^{\infty}(\R^N,\R^{m} \otimes \R^N) \to C^{\infty}(\R^N,\R^d) 
\]
is a potential of $\A$.
\end{rem}
\begin{prop} \label{prop:potentials}
Let $\A$ be a constant rank operator. The following statements are equivalent: \begin{enumerate}[label=(\alph*)]
			\item \label{prop:1}$\B$ is the potential of $\A$;
            \item \label{prop:exact} The following two properties hold \begin{enumerate}
            \myitem[(b1)]  \label{prop:2} $\forall u \in L^2_{\av}(T_N,\R^d) \cap \ker \A$ there exists $v \in W^{k_{\B},2}(T_N,\R^m)$ such that $\B v =u$;
            \myitem[(b2)] \label{prop:3} $\forall v \in C^{\infty}(T_N,\R^m)$ we have $\A (\B v) =0$;
            \end{enumerate}
        \item \label{prop:adjoint} $\A^{\ast}$ is a potential of $\B^{\ast}$.
        \end{enumerate}
\end{prop}
\begin{proof}
\ref{prop:1} $\Leftrightarrow$ \ref{prop:adjoint} is just a purely algebraic calculation, using that $(\ker \A[\xi])^{\perp} = \Image \A^{\ast}[\xi]$. We just need to verify \ref{prop:1}  $\Leftrightarrow$ \ref{prop:exact} . To show that \ref{prop:1}  implies \ref{prop:exact} , note that we may write $u \in L^2 \cap C^{\infty}(T_N,\R^d)$ as \[
u(x)= \sum_{\lambda \in \Z^N} \hat{u}(\lambda) e^{-2 \pi i x \cdot \lambda}.
\]
Note that $u \in \ker \A$ if and only if $\hat{u}(\lambda) \in \ker \A[\lambda]$, i.e. $\hat{u}(\lambda) = \B[\lambda](\hat{v}(\lambda))$ for some $\hat{v}(\lambda) \in \R^d$. We may choose $v$ in the orthogonal complement of $\ker \B[\lambda]$. Hence, we may bound \[
\V \hat{v}(\lambda) \V \leq C \V \lambda \V^{-k} \hat{u}(\lambda) \]
and thus $v$, defined by \[
v(x)= \sum_{\lambda \in \Z^N} \hat{v}(\lambda) e^{-2 \pi i x \cdot \lambda},
\]
is bounded in $W^{k_{\B},2}(T_N,\R^d)$ and satisfies $\B v= u$. $\A \circ \B =0$ follows by a calculation using the Fourier transform. The converse direction that \ref{prop:exact}  implies \ref{prop:1}  follows a very similar argument.

\end{proof}
\begin{rem}
Note that condition \ref{prop:3} can be generalised to domains $\Omega \subset \R^N$ in general, i.e. \ref{prop:3} holds for any $v \in C^{\infty}(\Omega,\R^m)$ if $\B$ is a potential of $\A$. Condition \ref{prop:2} cannot be extended to general domains $\Omega$. In particular, this holds for the pair $\B=\nabla$, $\A=\curl$ on a set $\Omega$ instead of $T_N$ only if $\Omega$ is simply connected.
\end{rem}

\textsc{Rai\c{t}\u{a}} showed the important equivalence of constant rank condition and existence of a potential \cite{Raita,Adolfo}.
\begin{prop} \label{Raita}
Let $\A$ be a homogeneous differential operator as in (\ref{A1}). $\A$ satisfies the constant rank property if and only if $\A$ admits a potential $\B$.
\end{prop}

\subsection{$\mathcal{A}$-quasiconvexity}
\begin{defi}
Let $\A$ be a differential operator and $f: \R^d \to \R$ a measurable function. $f$ is called $\A$-quasiconvex if $\forall x \in \R^d$, $\phi \in C_{\av}^{\infty} (T_N,\R^d)$ with $\A \phi =0$ we have \begin{equation}
f(x) \leq \int_{T_N} f(x + \phi(y)) \dy.
\end{equation}
We call $f$ $\B$-potential-quasiconvex if $\forall x \in \R^d$, $\forall \psi \in C_0^{\infty}(\Omega,\R^d)$ we have \begin{equation}
f(x) \leq \frac{1}{\V \Omega \V} \int_{\Omega} f(x + \B \psi(y)) \dy.
\end{equation}
\end{defi}
\begin{defi}
Let $\Lambda \subset \R^d$ be a cone, i.e. $t \in \R_+$, $v \in \Lambda$ $\Rightarrow~tv \in \Lambda$. We call $f: \R^d \to \R$ $\Lambda$-convex, if $\forall x \in \R^d$, $v \in \Lambda$ the function \begin{displaymath}
t \to f(x + tv)
\end{displaymath}
is convex. We call $f: \R^d \to \R$ $\Lambda$-affine if the above map is affine ($f$ is $\Lambda$-convex and $-f$ is $\Lambda$-convex).
\end{defi}

\begin{prop} \label{prop:equivalence} Let $\A$ be a homogeneous differential operator satisfying the constant rank property and $\B$ a  potential of $\A$. Let $f: \R^d \to \R$ be a continuous function. The following statements are equivalent \begin{enumerate}
\item $f$ is $\A$-quasiconvex.
\item $f$ is $\B$-potential quasiconvex.
\item Let $Q=(0,1)^N$. Then for all $\phi \in C^{\infty}_c(Q,\R^m)$ and for all $x \in \R^d$ we have \begin{equation}
f(x) \leq \int_Q f(x + \B \phi(y)) \dy = 0.
\end{equation}
\item For all $\phi \in C^{\infty}(T_N,\R^m)$ and for all $x \in \R^d$ we have \begin{equation}
f(x) \leq \int_{T_N} f(x + \B \phi(y)) \dy = 0.
\end{equation}
\end{enumerate}
\end{prop} 
A proof of this statement (in the setting $\B= \nabla$) can be found in \cite[Section 4.7]{Mlecture}. Let us also mention following statement about equivalence of $\A$-quasiconvexity and weak lower semicontinuity (c.f. \cite{FM}).
\begin{prop} \label{Awlsc}
Let $1 < p \leq \infty$.
Let $f: \R^d \to \R$ be continuous and  \begin{displaymath}
0 \leq f(x,v) \leq C(1+ \V v \V^p), \quad \text{if } p<\infty.
\end{displaymath}
Then the functional $I: L^p(\Omega,\R^d) \to [0,\infty]$, defined by \begin{displaymath}
I (u) = \left \{ \begin{array}{ll} \int_{\Omega} f(x,u(x)) \dy &\text{if } \A u = 0 \text{ in the sense of distributions,}\\ \infty & \text{else,} \end{array} \right. 
\end{displaymath}
is sequentially weakly lower-semicontinuous (weakly$*$ if $p= \infty$) if and only if $f(x,\cdot)$ is $\A$-quasiconvex for almost every $x \in \Omega$.
\end{prop}

%% file: quasiaffine2.tex
\section{$\mathcal{A}$-quasiaffine functions} \label{sec:affin}
In the following, $\A$ is a homogeneous differential operator of the form (\ref{A1}), satisfying the constant rank property and the spanning property. We denote by $\B$ its potential, which exists due to Proposition \ref{Raita}.
\begin{defi} \label{defi:Aqa} Let $M \in C(\R^d)$.\begin{enumerate} [label=(\alph*)]
\item We call $M$ $\A$-quasiaffine if $M$  and $-M$ are $\A$-quasiconvex.
\item We call $M$ $\B$-potential-quasiaffine if $M$ and $-M$ are $\B$-potential-quasiconvex.
\end{enumerate}
\end{defi}
\begin{prop} Let $M:\R^d \to \R$ be continuous and let $\B$ be a potential of $\A$. The following statements are equivalent: \begin{enumerate} [label=(\alph*)]
\item $M$ is $\A$-quasiaffine.
\item \label{alalala} For all $u,v \in L^p(T_N,\R^d) \cap \ker \A$ with $p \geq d$ and $\int_{T_N} u(y) \dy = \int_{T_N} v(y) \dy$ we have \begin{equation}
\int_{T_N} M(u(y)) \dy = \int_{T_N} M(v(y)) \dy.
\end{equation}
\item $M$ is $\B$-potential-quasiaffine.
\end{enumerate}
\end{prop}
This directly follows from Proposition \ref{prop:equivalence}; for the bound $p \geq d$ in \ref{alalala} we indeed also need Theorem \ref{prop:polynomial} \ref{prop:poly2} and \ref{prop:poly4}.

Note that we can already infer the following strong properties for $\A$-quasiaffine functions using basic methods. One key point is that, for a differential operator $\A$ satisfying the constant rank and the spanning property, $\A$-quasiaffine functions are already $\Lambda_{\A}$-affine.
\begin{thm}\label{prop:polynomial} \begin{enumerate} [label=(\alph*)]
\item \label{prop:poly1} Let $f: \R^d \to \R$ be $\A$-quasiconvex and continuous. Then $f$ is also $\Lambda_{\A}$-convex.
\item \label{prop:poly15} Let $f \in C^2(\R^d)$. Then $f$ is $\Lambda_{A}$-convex if and only if for all $x \in \R^d$ and $v \in \Lambda_{A}$ \[
D^2 f (x) [v,v] = \frac{\partial^2}{\partial t^2} f(x +tv) \geq 0.
\]
\item \label{prop:poly2} Let $M: \R^d \to \R$ be $\A$-quasiaffine and continuous. Then $M$ is also $\Lambda_{\A}$-affine.
\item \label{prop:poly25} Let $M \in C^2(\R^d)$. Then $f$ is $\Lambda_{A}$-affine if and only if for all $x \in \R^d$ and $v \in \Lambda_{A}$ \[
D^2 M (x) [v,v] = \frac{\partial^2}{\partial t^2} M(x +tv) = 0.
\]
\item \label{prop:poly3} Let $M: \R^d \to \R$ be a polynomial of degree 2. Then $M$ is $\A$-quasiaffine if and only if $M$ is $\Lambda_{\A}$-affine.

\item \label{prop:poly4}Any $\Lambda_{\A}$-affine map is a polynomial of degree $\leq d$.
\item \label{prop:poly5} Any partial derivative of a $\Lambda_{\A}$-affine map is also $\Lambda_{A}$-affine.
\item \label{prop:poly55} A homogeneous polynomial $M \colon \R^d \to \R$ of degree $\geq 3$ is $\Lambda_{\A}$-affine if all its partial derivatives $\partial_i M$ $i \in \{1,...,d\}$ are $\Lambda_{\A}$-affine.
\item \label{prop:poly6} There exists a basis of homogeneous polynomials of the space of $ \Lambda_{\A}$-affine maps.
\end{enumerate}
\end{thm}
\begin{proof}
\ref{prop:poly1} follows from considering test functions of the form $v e^{-2 \pi i \lambda x}$ for $v \in \ker \A[\lambda]$ (c.f. \cite{FM}). \ref{prop:poly15} follows from the classical equivalence of convexity and $f''(x) \geq 0$ for $f \in C^2(\R)$.  \ref{prop:poly2} and \ref{prop:poly25} then directly follow by Definition \ref{defi:Aqa}. \ref{prop:poly3} relies on Plancherel's identity, which is valid for quadratic forms. In particular, as all affine functions are automatically $\A$-quasiaffine, we may consider $M$ to be $2$-homogeneous. Then, using Plancherel's identity, we find that \[
\int_{T_N} M(u(y)) \dy =\sum_{\lambda \in \Z^N} M(\hat{u}(\lambda)).
\]
As $M$ is homogeneous of degree $2$ and $\hat{u}(\lambda) \in \Lambda_{\A}$, it follows that $M(\hat{u}(\lambda)) =0$ for $\lambda \neq 0$.

Ad \ref{prop:poly4}: Let now $v_1,...,v_d$ be a basis of $\R^d$, which is contained in $\Lambda_{\A}$. The existence of such a basis is ensured by the \textit{spanning property} for $\A$. Denote by $\lambda_1(y),...,\lambda_n(y)$ the coordinates with respect to this basis. We may write a $\Lambda_{A}$-affine function $f$ as\[
f(y) = \tilde{f} (\lambda_1,...,\lambda_d).
\]
Due to $\Lambda_{\A}$-affinity, we know that the map \[
\lambda_i \mapsto \tilde{f}(\lambda_1,...,\lambda_d). 
\]
is affine for fixed $i \in \{1,...,d\}$ and fixed $\lambda_j$ $j \neq i$ . Hence, $\tilde{f}$ must be a polynomial in $\lambda_i$. In particular, as $\tilde{f}$ is affine in each $\lambda_i$, it has at most degree $d$.

The property \ref{prop:poly5} follows from \ref{prop:poly25}. To see \ref{prop:poly55} note that \begin{align*}
   D^2M(x)[v,v] &= \int_0^1 D^3M(tx)[v,v,x] \dt + D^2 M(0)[v,v]\\&= \int_0^1 D^2 \left( \frac{\partial}{\partial x}M \right)(tx) [v,v]+ D^2M(0)[v,v]. 
\end{align*}

As $M$ is homogeneous of degree strictly larger than two, $D^2M(0) =0$ and therefore $M$ is $\A$-quasiaffine.

For \ref{prop:poly6} we use \ref{prop:poly55}. Write $f= \sum_{i=1}^d f_i$ for $i$-homogeneous polynomials $f_i$. We may consider $\tilde{f}= f - f_0 - f_1$, as $f_0$ and $f_1$ are affine and hence $\Lambda_{\A}$-affine. Observe that then $\Lambda_{\A}$-affinity yields $f (x) =0$ for all $x \in \Lambda_{\A}$. In particular, $f_i(x)=0$ for all $i=2,...,d$ and $x \in \Lambda_{\A}$. 

\noindent But this implies  $\Lambda_{\A}$-affinity for $f_2$. Considering $\bar{f} = \nabla (f- f_0-f_1-f_2)$, the statement \ref{prop:poly55} and an inductive argument, we get that $f_0,...,f_d$ are all already $\Lambda_{\A}$-affine. Therefore, there must be a basis of homogeneous polynomials for $\Lambda_{\A}$-affine maps.
\end{proof}

\begin{rem} \begin{enumerate}[label=\alph*)]
    \item Due to Theorem \ref{prop:polynomial} \ref{prop:poly55}, if there is $\Lambda_{A}$-affine polynomial $f$ of degree $k$, then there is also a $\Lambda_{\A}$-affine polynomial of degree $k-1$. In particular, the question of existence of non-affine $\Lambda_{\A}$-affine functions reduces to the existence of quadratic $\Lambda_{\A}$-affine functions. Recall that $\A$-quasiaffine functions are $\Lambda_{\A}$-affine functions and the converse holds for quadratic functions. Hence, the existence of non-trivial $\A$-quasiaffine functions reduces to the existence of a quadratic function vanishing on $\Lambda_{\A}$.
    \item \textsc{\v{S}ver\'{a}k} showed in \cite{Sverak}, that the other direction in \ref{prop:polynomial} \ref{prop:poly1} is not true, i.e. there exist $\Lambda_{\A}$-convex functions that are not $\A$-quasiconvex.
    \item  The converse implication in \ref{prop:polynomial} \ref{prop:poly2} is false, i.e. $\Lambda_{\A}$-affinity does not imply $\A$-quasiconvexity (c.f. Lemma \ref{lem:counter}, \cite{BCO}).
\end{enumerate}
\end{rem}
Let us now proof Theorem \ref{thm:A}. For this, we roughly follow the proof of this statement for $\B= \nabla^k$ in \cite{BCO}. Note that by considering the potential $\B$ of $\A$, (b) in Theorem \ref{thm:A} is equivalent to 
\begin{itemize}
    \myitem[(b')] \label{mt:2} $M$ is a polynomial and $\forall x \in \R^d$, $\forall  r \geq 2$, $\forall \xi_1,...,\xi_r \in \R^d$ which are linear dependent and $ \forall w_1,...,w_r \in \R^m$ we have \begin{equation} \label{Aqa}
D^r M (x) [\B[\xi_1](w_1),...,\B[\xi_r](w_r)] =0 .
\end{equation}
\end{itemize}

\begin{proof}[Proof of Theorem \ref{thm:A}] The validity of the implications $\ref{thmA:1} \Leftrightarrow \ref{thmA:4} \Leftrightarrow \ref{thmA:5}$ follows from Theorem \ref{Awlsc}. Indeed, for $\phi \in C_c^{\infty}(\R^N,\R^d)$ (or $\phi \in L^1$ for \ref{thmA:4}), consider 
\[
I_{\phi}(u) = \begin{cases} \int_{\Omega} M( u(x)) \cdot \phi(x) \dx & \text{if } \A u =0 \\ \infty & \text{ else.}
\end{cases}
\]
This functional is weakly continuous if and only if, for all $x \in \R^d$, the map $v \mapsto M(v) \cdot \phi(x)$ is $\A$-quasiaffine, which is equivalent to $\A$-quasiaffinity of $M$ (for more details we refer to \cite{RG}).

We now prove $\ref{thmA:1} \Leftrightarrow \ref{thmA:3}$. If $M$ is $\B$-potential-quasiaffine, then by Theorem \ref{prop:polynomial}, it is a polynomial and hence it is even $C^{\infty}$. Moreover, for all $u \in C^k(\bar{\Omega})$ and all $\phi \in C^{\infty}_c(\Omega,\R^m)$ we have \begin{equation} \label{calcEL}
\begin{aligned}
0& = \frac{d}{dt}\left( \int_{\Omega} M(\B u(y) + t \B \phi(y)) \dy \right)_{\V t=0} \\
~& = \int_{\Omega} \frac{d}{dt} \left( M(\B u(y) +  t \B \phi(y)) \right) _{\V t=0}  \dy \\
~& = \int_{\Omega} DM (\B u(y)) \cdot \B \phi(y) \dy.
\end{aligned}
\end{equation}
Thus, \eqref{EL:intro} holds in the sense of distributions if $M$ is $\B$-potential-quasiaffine. The same calculation as in \eqref{calcEL} also shows that, if \eqref{EL:intro} holds, then $M$ is $\A$-quasiaffine.

It remains to show that $\ref{thmA:1} \Leftrightarrow \ref{mt:2}$. First, we prove the direction $\ref{thmA:1} \Rightarrow \ref{mt:2}$.

If $r=2$, note that $\B[\lambda \xi] = \lambda ^{k_{\B}} \B[\xi]$ for $\xi \in \R^N$ and $\lambda \in \R \back \{0\}$. Hence, if $\xi_1$ and $\xi_2$ are linear dependent and nonzero, we may write $\xi_2 = \lambda \xi_1$ and \begin{displaymath}
\B[\xi_2](w_2) = \B[\xi_1] (\lambda^{k_{\B}} w_2).
\end{displaymath}
Therefore, we may only consider $\xi_1=\xi_2 = \xi$. Thus,
\begin{displaymath}
\begin{aligned} 
D^2 M(x) [v_1,v_2] &= D^2 M(x) [\B[\xi](w_1),\B[\xi](w_2)] \\
&= \frac{1}{2}  D^2 M(x) [\B[\xi](w_1+w_2),\B[\xi](w_1+w_2)]\\ &- \frac{1}{2}D^2 M(x) [\B[\xi](w_1),\B[\xi](w_1)] -\frac{1}{2} D^2 M(x) [\B[\xi](w_2),\B[\xi](w_2)] \\ 
&=0.
\end{aligned}
\end{displaymath}
We prove the statement for $r >2$ by induction. Let \eqref{Aqa} hold for some $r \in \N$. We consider linear dependent $\xi_1,...,\xi_{r+1} \in \R^N$ and $w_1,...,w_{r+1} \in \R^m$. First, suppose that $\xi_1,...,\xi_{r}$ are already linear dependent. Then by induction hypothesis,  \begin{displaymath}
D^r M(x) [\B[\xi_1](w_1),...,\B[\xi_r](w_r)] =0 \quad \forall x \in \R^d.
\end{displaymath}
Taking the derivative in direction $\Bop(\xi_{r+1}) (w_{r+1})$, the result is also $0$. Hence, \begin{displaymath}
D^{r+1} M(x) [\B[\xi_1](w_1),...,\B[\xi_{r+1}](w_{r+1})] =0 .
\end{displaymath}
We may suppose that $\xi_{r+1}$ can be written as a linear combination of linear independent $\xi_1,...,\xi_r \in \R^N \back \{0\}$. Due to the homogeneity of $\B[\cdot](w)$ we  may also suppose that \begin{displaymath}
\xi_{r+1} = \xi_1 + ... + \xi_r.
\end{displaymath}
Let $t_1,...,t_r \in \R$ be real parameters. Define the function $\phi \in C^{\infty}(T_N,\R^m)$ by \begin{displaymath}
\phi (y) := \begin{cases} \sum_{i=1}^{r+1} t_i w_i \cos(2 \pi \xi_i \cdot y) & \text{if } k_{\B} \text{ is even}, \\
                             \sum_{i=1}^{r+1} t_i w_i \sin(2 \pi \xi_i \cdot y) & \text{if } k_{\B}  \text{ is odd.} \end{cases}
\end{displaymath}
For the sake of simplicity we shall consider the case $k_\B =2k$, the other case is rather similar.

Then, $\B \phi$ is given by
\begin{displaymath}
\B \phi(y) = (-4\pi^2)^k \sum_{i=1}^{r+1} t_i \B[\xi_i](w_i) \cos(2 \pi \xi_i \cdot y).
\end{displaymath}
Now, $\B$-potential-quasiafffinity means that \begin{equation} \label{eq:pe}
\int_{T_N} M(x + \B \phi) \dy = M(x)   \quad \forall x \in \R^d.
\end{equation}
The left-hand side of \eqref{eq:pe} is a polynomial in $t_i$. The coefficient of $t_1 \cdot...\cdot t_{r+1}$ is the constant $(-4 \pi^2)^k$ times
\begin{align*} 
\int_{T_N} D^{r+1}M(x)& [\B[\xi_1](w_1),...,\B[\xi_{r+1}](w_{r+1})] \cdot \cos(2 \pi \xi_1 \cdot y) \cdot ... \cdot \cos(2 \pi \xi_{r+1} \cdot y)\dy \\
&=D^{r+1}M(x) [\B[\xi_1](w_1),...,\B[\xi_{r+1}](w_{r+1})] \\ &\quad \cdot  \int_{[0,1]^N} \cos( 2 \pi \xi_1 \cdot y) \cdot ... \cdot \cos(2 \pi \xi_{r} \cdot y) \cos(2 \pi \sum_{i=1}^r \xi_{i} \cdot y) \dy  \\
&=2^{-r} D^{r+1}M(x) [\B[\xi_1](w_1),...,\B[\xi_{r+1}](w_{r+1})].
\end{align*}
To calculate the integral in this equation, we just use the addition theorem for $\cos$ and Fubini. As the coefficient of $t_1 \cdot ... \cdot t_{r+1}$ is $0$ on the right-hand side of \eqref{eq:pe}, we get the desired result.

For the direction $\ref{mt:2} \Rightarrow \ref{thmA:1}$ we first claim that it suffices to show that $\forall x \in \R^d$, $\forall \phi \in C^{\infty}(T_N,\R^m)$ and for all $r \geq 2$ \begin{equation} \label{claimmt}
 \int_{T_N} D^r M(x) [\B \phi (y),...,\B \phi (y)] = 0.
\end{equation}
Suppose that (\ref{claimmt}) holds. We want to show \ref{thmA:1}. Take arbitrary $x \in \R^d$ and $\phi \in C^{\infty}(T_N,\R^m)$. Consider the Taylor series of $M$ at the point $x$ in the direction of $\B \phi (y) \in \R^d$. As $M$ is a polynomial of some degree $s$, $M$ equals its Taylor polynomial in $x$ of degree $s$, i.e. \begin{displaymath}
M(x+\B \phi(y)) = \sum_{r=0}^s \frac{1}{r!} D^r M(x) [\B \phi(y),...,\B \phi(y)]
\end{displaymath}
Integrating over $y \in T_N$, using \ref{mt:2} and the fact that $\B \phi$ has average $0$, yields \begin{align*}
\int_{T_N} M(x + \B \phi (y)) \dy &= \sum_{r=0}^s \int_{T_N}  \frac{1}{r!} D^r M(x) [\B \phi(y),...,\B \phi(y)] \dy \\
&= \int_{T_N} M(x) \dy + \int_{T_N } DM(x) \cdot \B \phi(y) \dy  \\
& \quad \quad + \sum_{r=2}^s \int_{T_N}  \frac{1}{r!} D^r M(x) [\B \phi(y),...,\B \phi(y)] \dy \\
&= \int_{T_N} M(x) \dy = M(x).
\end{align*}
Thus, it suffices to prove (\ref{claimmt}). To this end, we use the following formula: \\
If $f_1,...f_r \in C^{0}(T_N,\R)$, then \begin{equation} \label{Plancherel}
\int_{T_N} f_1(y) \cdot  ... \cdot f_r(y) \dy = \sum_{\xi_1,...,\xi_{r-1}\in \Z^N} \overline{\hat{f}_1(\xi_1)} \cdot \hat{f_2}(\xi_2) \cdot .... \cdot \hat{f_{r-1}}(\xi_{r-1}) \cdot \hat{f_r}\left(\xi_1 - \sum_{i=2}^{r-1} \xi_i \right).
\end{equation}
This equation can be derived by using Plancherel's theorem once for $f_1$ and $f_2\cdot ... \cdot f_r$ and then using a discrete version of the convolution  formula, i.e. \begin{displaymath}
\widehat{(f(\cdot) g(\cdot))} (\xi_1) = \sum_{\xi_2 \in \Z^n} \hat{f}(\xi_2) \cdot \hat{g} (\xi_1-\xi_2).
\end{displaymath}

Recall that $D^r M(x)[\cdot,...,\cdot]$ is a multilinear form (i.e. a homogenenous polynomial  in the entries). Therefore, we can use the identity \eqref{Plancherel}. Hence \begin{align*}
&\int_{T_N} D^r M(x) [\B \phi (y),...,\B \phi (y)] \\
&=\sum_{i=1}^{r-1} \sum_{\xi_i \in \Z} D^r M(x) \left[ \B[\xi_1] (\hat{\phi}(\xi_1)),...,\B[\xi_{r-1}](\hat{\phi}(\xi_{r-1})), \B[\xi_1-\sum_{i=2}^{r-1} \xi_i] (\hat{\phi}(\xi_1 - \sum_{i=2}^{r-1} \xi_i)) \right] \\
&= 0,
\end{align*}
as the vectors \[\xi_1,~...,~\xi_{r-1},~\xi_1 - \sum_{i=2}^{r-1} \xi_i \]
are linear dependent. Each summand equals $0$ due to condition \eqref{Aqa} in \ref{mt:2}. We have shown the claim and therefore that \ref{mt:2} implies \ref{thmA:1}.
\end{proof}
\begin{rem}
It shall be remarked, that Theorem \ref{thm:A} shows that $\A$-quasiaffinity is a local property; it can be verified by only considering gradients of $f$ pointwise. As it is shown in \cite{Kristensen}, this is not true for $\A$-quasiconvexity.
\end{rem}
Let us now prove that the condition \ref{mt:2}, which is equivalent to Theorem \ref{thm:A} \ref{thmA:2}, can be slightly weakened to the following:

\begin{thm} \label{main:2}
Let $\B \colon C^{\infty}(\R^N,\R^m) \to C^{\infty}(\R^N,\R^d)$ be a constant rank operator of order $k_{\B}$. Then $M \colon \R^d \to \R$ is $\A$-quasiaffine if and only if for all \[ 2 \leq r \leq \min \{ k_{\B}, N\} +1,\] $\xi_1,...\xi_r \in \R^N \back\{0\}$ linear dependent and $w_1,...,w_r \in \R^m$ we have  \begin{equation} \label{Aqa2}
D^{r} M(x) [\B[\xi_1](w_1),...,\B[\xi_{r}](w_{r})] =0.
\end{equation}
\end{thm}
\begin{rem} If $k_{\B}=1$, then \eqref{Aqa2} only needs to hold for $r=2$. In this case, \eqref{Aqa2} is equivalent to $\Lambda_{\A}$-affinity. Hence, as a special case of Theorem \ref{main:2}, we get the statement of Theorem \ref{thm:B}, that for first-order potentials, $\A$-quasiaffinity and hence $\B$-potential quasiaffinity are equivalent to $\Lambda_{\A}$-affinity.
\end{rem}

\begin{proof}
We just need to prove that if equation \eqref{Aqa2} is true for $2 \leq r \leq \min \{ k_{\B}, N\} +1$, then it also holds for $r \in \N$. Let us first deal with the case $\min \{ k_{\B}, N\} = N$. Note that then for $j>2$ and $r= N+ j$, there are $N+1$ vectors $\xi_i$, which are already linear dependent; say $\xi_1,...,\xi_{N+1}$ are linear dependent. Then, 
\[
D^{N+1} (x) [\B[\xi_1](w_1),...,\B[\xi_{N+1}](w_{N+1})] =0.
\]
Therefore, also 
\[
D^{N+j} (x) [\B[\xi_1](w_1),...,\B[\xi_{N+j}](w_{N+j})] =0.
\]

\medskip Suppose now that $k_{\B} \leq N$. If $k_{\B} =1$, then for all $\xi_1,\xi_2 \in \R^N \back \{0\}$ and $w \in \R^m$ \[
\B[\xi_1 + \xi_2](w)= \B[\xi_1](w) +\B[\xi_2](w) \in \spann \{ \B[\xi_1] (w) , \B[\xi_2](w) \}.
\]
We prove an analogue of this statement for $k_{\B} > 1$. Again, make the reductions from the proof of Theorem \ref{thm:A}. We just need to show that, for $r > k_{\B}+1$, $\xi_1,...,\xi_{r-1} \in \R^N \back \{0\}$ linear independent and  $w_1,...,w_r \in \R^m$, we have \begin{equation}\label{Claim}
D^r M (x) \left[\B[\xi_1](w_1),...,\B[\xi_{r-1}](w_{r-1}),\B[\xi_1+...+\xi_{r-1}](w_r)\right]  =0 .
\end{equation}
\textbf{Claim:}\textit{ For $w \in \R^m$ \begin{equation} \label{eq:span}
    \B\left[\sum_{i=1}^{r-1} \xi_i\right](w) \in \spann_{\lambda \in I} \left\{ \B \left[\sum_{i=1}^{r-1} \lambda_i \xi_i\right](w) \right\},
\end{equation}
where $r >k_{\B}+1$ and the set $I$ of coefficients  is given by \[
I= \big\{ \lambda \in \R^{r-1} \colon \lambda_i =0 \text{ for some } i \in \{1,...,r-1\} \big\}
\]}

\medskip Suppose that \eqref{eq:span} is proven. Then, for a finite index set $J \subset I$, we can write, \[
\B\left[\sum_{i=1}^{r-1} \xi_i\right](w) = \sum_{\lambda \in J} \B\left[\sum_{i=1}^{r-1} \lambda_i \xi_i\right](w) \]
and use that, for each $\lambda \in J$, there is $i \in \{1,...,r-1\}$ such that $\lambda_i =0$. W.l.o.g. $i=1$ for some fixed $\lambda \in J$. Then \begin{align*}
&D^r M(x) \left [\B[\xi_1](w_1),...,\B[\xi_{r-1}](w_{r-1}), \B[\sum_{i=2}^{r-1} \lambda_i \xi_{r-1}](w_r)\right] \\
&= \frac{\partial}{\partial_{ \B[\xi_1](w_1)}} D^{r-1} M(x) \left[\B[\xi_2](w_2),...,\B[\xi_{r-1}](w_{r-1}), \B[\sum_{i=2}^{r-1} \lambda_i \xi_{r-1}](w_r)\right].
\end{align*}
Note that we assume in \ref{Aqa2} that the right-hand side is $0$, whenever $r-1 \leq k_{\B}+1$, i.e. $r \leq k_{\B}+2$. Assuming that \eqref{eq:span} holds, one can prove \eqref{Claim} for all $r \in \N$ by an inductive argument.

It remains to prove the validity \eqref{eq:span}. Consider for $t_1,...,t_{r-1}$ the polynomial \[
P (t_1,...,t_{r-1}) =\B \left[\sum_{i=1}^{r-1} t_i \xi_i \right](w_r).
\]
This polynomial has degree $k_{\B} < r-1$. Hence, in every monomial of $P$ of the form $\prod_{i=1}^{r-1} t_i^{\alpha_i}$ there is at least one $j \in \{1,...,r-1\}$, such that $\alpha_j=0$. But we can recover the coeffients of these monomials by considering \[
\B\left[\sum_{i=1,i\neq j}^{r-1} t_i \xi_i\right](w_r).
\]
In particular, we can recover these coefficients by taking linear combinations of $P(\lambda)$ for suitable $\lambda \in \{ \mu \in \R^{r-1} \colon \mu_j =0\} \subset I$. Therefore, \eqref{eq:span} holds. This concludes the proof of Theorem \ref{main:2}.
\end{proof}
Theorem \ref{thm:B} is a special case of Theorem \ref{main:2}. In this setting, $k_{\B}=1$, i.e. $\A$-quasiaffinity of $M$ is equivalent to the fact that \[
D^2 M(x) [\B[\xi](w_1),\B[\xi](w_2)] =0.
\]
As it was already established in Theorem \ref{prop:polynomial} \ref{prop:poly25}, this is indeed equivalent to $\Lambda_{\A}$-affinity of $M$.

Let us recall the \textsc{Ball-Currie-Olver} example showing that $\A$-quasiaffinity does \textit{not} follow if \eqref{Aqa2} does not hold for all $2 \leq r \leq \min\{k_{\B},N\} +1$ (cf. \cite{BCO}). Let us consider the setting $k_{\B}=2$.

\begin{lem}[Ball, Currie, Olver] \label{lem:counter} There is a first-order differential operator $\A$ and a map $L: \R^d \to \R$ which is  $\Lambda_{\A}$-affine, but not $\A$-quasiaffine.
\end{lem}
\begin{proof} Consider the differential operator $\B = \nabla ^2$, i.e. 
\begin{displaymath}
(\nabla^2 u)_{ijk} = \partial_i \partial_j u_k (i,j=1,...,N; k=1,...,m) 
\end{displaymath} and $\A$ the corresponding first order operator, such that $\B$ is a potential of $\A$ (cf. \cite{Meyers}).
The characteristic cone of $\A$ is the space of tensors of the form \begin{displaymath}
\lambda \otimes \lambda \otimes b \colon \lambda \in \Sphere^{N-1},~b \in \R^m.
\end{displaymath}
Now choose $N=2$ and $m=3$ and consider the map $L$ defined via \begin{equation}
L(\nabla^2 u) = \sum \limits_{\sigma \in S_3} \sgn(\sigma) \partial_x^2 u_{\sigma(1)} \partial_x \partial_y u_{\sigma(2)} \partial_y^2 u_{\sigma(3)}.
\end{equation}
One can check that this is affine in $\Lambda_{\A}$. On the other hand, one can check that ,for \begin{displaymath}
u(x_1,x_2) = \left( \begin{array}{c} \cos(2 \pi x_1) \\ \cos(2 \pi x_2) \\ \cos(2 \pi(x_1+x_2)), \end{array} \right)
\end{displaymath}
we have \begin{displaymath}
\int_{T_N} L(u(x_1,x_2)) \dx = -\frac{1}{4}.
\end{displaymath}
\end{proof}
\noindent We have seen in Theorem \ref{main:2}, that the answer to the question, whether \[
f~\Lambda_{A}\text{-convex} \Longrightarrow f~\A \text{-quasiaffine} \]
really depends on the order of the operator $k_{\B}$. We note that in view of the following lemmata, the minimal order of $k_{\B}$ of the potential $\B$ cannot be bounded in terms of the order of $\A$. In view of Theorem \ref{main:2}, the differential condition on $M$ for being $\A$-quasiaffine therefore depends only on the order of $\B$ and \textit{not} on the order of $\A$.
\begin{lem} \label{lem:ho}
Let $\B \colon C^{\infty}(\R^2,\R^m) \to C^{\infty}(\R^2,(\R^2)^k)$ be a differential operator such that  \[
\Image \B[\xi] = \Image \nabla^k [\xi]\quad \forall \xi \in \R^N \back \{0\},
\]
where $\nabla^k \colon C^{\infty}(\R^2,\R) \to C^{\infty}(\R^2,(\R^2)^k)$. Then the operator $\B$ is of order $k_{\B} \geq k$.
\end{lem}

\begin{proof}
We note that \[
\dim (\Image \nabla^k[\xi])=1. \]
Consider $\xi_0=e_1+e_2$ and the coordinates of \[
v_{1^k} = \partial_1^k u, \quad v_{2^k} = \partial_2^k u.
\]
There exist $w \in \R^m$, such that \[
\B [\xi_0](w) \neq 0, (\B [\xi_0](w))_{1^k} =1 = (\B [\xi_0](w))_{2^k}=1.
\]
Due to continuity of $\B[\cdot](w)$, there exists an open ball $B_r(\xi_0)$, such that, for all $\xi \in B_r(\xi_0)$,
\[ \B [\xi](w) \neq 0. \]
In particular, as the dimension of the image of $\nabla^k[\xi]$ (and therefore also of the image of $\B[\xi]$) is one, we then have, for all $\xi \in B_r(\xi_0)$,
\[ \xi_2^k (\B[\xi](w))_{1^k} = \xi_1^k (\B[\xi](w))_{2^k} \]
Hence, $(\B[\xi](v))_{1^k}$ and $(\B[\xi](v))_{2^k}$ are polynomials of degree larger than $k$ in $\xi$. Therefore, $\B$ has at least order $k$.
\end{proof}

\begin{coro}
Let $N>2$. \begin{enumerate}[label=(\alph*)]
    \item \label{coro:1} For any $k \in \N$, there exists a first-order operator $\A$, such that any potential $\B$ of $\A$ has order $k_{\B}\geq k$.
    \item \label{coro:2} For any $k \in \N$, there exists a first-order operator $\B$, such that any \textit{annihilator} $\A$ of $\B$ (i.e. an operator $\A$, such that $\B$ is a potential of $\A$) has order $k_{\A}\geq k$.
\end{enumerate}
\end{coro}
Note that \ref{coro:1} follows directly from Lemma \ref{lem:ho} and the result by \textsc{Meyers}, that $\nabla^k$ admits a first-order annihilator $\A^k$ \cite{Meyers}. \ref{coro:2} then follows from Proposition \ref{prop:potentials} \ref{prop:adjoint}. In particular,  $\B = (\A^k)^{\ast}$ is of first order and only admits annihilators of order $\geq k$.

%% file: quasiaffine3.tex
\section{Examples for $\A$-quasiaffine functions} \label{sec:examples}

In this section we discuss some examples and results for $\A$-quasiaffine (or $\B$-potential-quasiaffine functions) for explicit $\A$.

\subsection{$\B=\nabla$ and $\B = \nabla^k$}
Consider the operator \[\B = \nabla \colon C^{\infty}(\R^N,\R^{m}) \longrightarrow C^{\infty}(\R^N,\R^{N \times n}),\] which is given by the coordinates \[
(\B u)_{ij} = \partial_i u_j.
\]
We have the following result in this special case (e.g. \cite{Mor66,Resh,Conti,Dac}). 
\begin{prop}
$M \colon \R^{N \times m} \to \R$ is a $\nabla$-potential-quasiaffine function if and only if it is a linear combination of $r \times r$ minors ($1 \leq r \leq \min\{ m , N\}$).
\end{prop}
As $\curl$ is the annihilator of $\nabla$, we therefore have a basis of $\curl$-quasiaffine functions.
For higher-order gradients ($\B = \nabla^k$), a characterisation is given in \cite{BCO}. Namely,  a basis of $\nabla^k$-potential-quasiaffine functions are already $\nabla$-potential quasiaffine functions for the gradient acting on $C^{\infty}(\R^N,\R^N \otimes_{sym} ...\otimes_{sym} \R^N) $. 

\subsection{The divergence operators on matrices}
Consider the divergence operator acting on matrices, i.e. \[\divergence \colon C^{\infty}(\R^N,\R^{N \times d}) \longmapsto C^{\infty}(\R^N,\R^d)\] defined by \[
(\divergence u)_i = \sum_{j=1}^N \partial_j u_{ij}.
\]
As, after a suitable rotation, the differential operator $\divergence$ equals $\curl$ for dimension $N=2$, we just consider dimension $N\geq 3$. Note that the characteristic cone $\Lambda_{\divergence}$ is the space of matrices with rank $\geq N-1$.

\begin{prop} \label{prop:rank2}
Let $N>2$ and $\Lambda_2 \subset \R^N \times \R^d$ be the set of $\R^{N \times d}$ matrices with rank less or equal to $2$. Then $f$ is $\Lambda_2$-affine if and only if $f$ is affine. Moreover, all $\divergence$-quasiaffine functions are already affine.
\end{prop}
\begin{proof}
Let $e_{ij}$ be the standard basis of $\R^{N \times d}$ matrices.
We only prove that $\Lambda_2$-affine functions $M$ are already affine, the second follows directly from the observation, that these are contained in the characteristic cone of the divergence.
To this end, note that $\Lambda_2$-affinity of $M$ implies that $M$ is a polynomial and is the sum of some monomials $P$. In particular, if a matrix $A$ is decomposed as $A= \sum_{i,j} A_{ij} e_{ij}$ for $A_{ij} \in \R$, $M(A)$ is a polynomial in $A_{ij}$.
Consider some matrix $B \in \R^{N \times d}$ and the directions \[ \lambda e_{kl} + \mu e_{ij}, \quad 1 \leq i,k \leq N, \quad 1 \leq l,j \leq d \]
for $\lambda, \mu \in \R$. The  map 
\[ 
(\lambda, \mu) \mapsto  M(B + (\lambda e_{kl} + \mu e_{ij})) \]
is affine. Hence, all coefficients containing of monomials $P$, such that $A_{ij}^2$ or $(A_{ij} A_{kl})$ are factors of $P$, vanish. Consequently, all coefficients of monomials with degree larger than one vanish. Therefore, $M$ is already affine.
\end{proof}

\subsection{The div-curl Lemma and similar operators} \label{sec:divcurl}
Consider a constant rank operator $\A_1 \colon C^{\infty}(\R^N,\R^d) \to C^{\infty}(\R^N,\R^l)$ and a potential $\B_1  \colon C^{\infty}(\R^N,\R^m) \to C^{\infty}(\R^N,\R^d)$. Then we may consider the operator $\A:=(\A_1,\B_1^{\ast}) \colon C^{\infty}(\R^N,\R^d \times \R^d) \to C^{\infty}(\R^N,\R^l \times \R^m)$ defined by \[
(\A_1,\B_1^{\ast}) (u,v) = (\A_1 u , \B_1^{\ast} v).
\]
Note that we have \[
(\ker \A_1 [\xi])^{\perp} = \ker \B_1^{\ast} [\xi] \quad \forall \xi \in \R^N \back \{0\}.
\]
Therefore, we obtain the following result.
\begin{prop}
The map $f \colon \R^d \times \R^d \to \R$ defined by \[
f(a,b) = a \cdot b \]
is an $(\A_1,\B_1^{\ast})$-quasiaffine map.
\end{prop}

Note that this result has a lot of implications for weak convergence in the context of compensated compactness (e.g. \cite{Murat2,Murat,Compcomp,DiPerna,Rindler,RG}). In particular, if $u_n, v_n \in L^2(\Omega,\R^d)$ with $\A_1 u_n=0$ and $\B_1^{\ast} v_n=0$, then due to the Characterisation Theorem \ref{thm:A} \ref{thmA:5} we find that \[
u_n \weakto u, ~ v_n \weakto v\quad \Longrightarrow \quad u_n \cdot v_n \weakto u \cdot v \text{ in the sense of distributions.}
\]
The two most prominent examples are the following. On the one hand, the operators \[
\A_1 = \curl, \quad \B_1^{\ast} = \divergence,
\]
both acting on $\R^{N \times d}$-matrices are a well-known example, which is the initial example of compensated compactness. Another example of this type is \[
\A_1 = \curl \curl^T, \quad \B_1^{\ast}=\divergence\]
on symmetric $N \times N$ matrices, considered in the context of linear elasticity (e.g. \cite{CMO}.)

\textbf{Acknowledgements:} The author would like to thank Stefan M\"uller for some helpful advice.  The author has been supported by the Deutsche Forschungsgemeinschaft (DFG, German Research Foundation) through the graduate school BIGS of the Hausdorff Center for Mathematics (GZ EXC 59 and 2047/1, Projekt-ID 390685813).

%% file: Aquasiaffine.bbl
\begin{thebibliography}{CDKS06}

\bibitem[AD92]{Alibert}
J.~Alibert and B.~Dacorogna.
\newblock An example of a quasiconvex function not polyconvex in dimension two.
\newblock {\em Arch. Rat. Mech. Anal.}, 117:155--166, 1992.

\bibitem[ARS21]{Adolfo}
A.~Arroyo~Rabasa and J.~Simental.
\newblock An elementary proof of the homological properties of constant-rank
  operators.
\newblock {\em https://arxiv.org/abs/2107.05098}, 2021.

\bibitem[Bal77]{Ball77}
J.~Ball.
\newblock Convexity conditions and existence theorems in non-linear elasticity.
\newblock {\em Arch. Rat. Mech. Anal.}, 63:337--403, 1977.

\bibitem[BCO81]{BCO}
J.~Ball, J.~Currie, and P.~Olver.
\newblock {Null Lagrangians, Weak Continuity, and Variational Problems of
  Arbitrary Order}.
\newblock {\em J. Func. Anal.}, 41:135--174, 1981.

\bibitem[CDKS06]{Conti}
S.~Conti, G.~Dolzmann, B.~Kirchheim, and M\"uller S.
\newblock {Sufficient conditions for the validity of the Cauchy-Born rule close
  to $\rm SO(n)$.}
\newblock {\em Journal of the European Mathematical Society}, 8:515--530, 2006.

\bibitem[CMO18]{CMO}
S.~Conti, S.~M\"uller, and M.~Ortiz.
\newblock Data-driven problems in elasticity.
\newblock {\em Arch. Rat. Mech. Anal.}, 229:79--123, 2018.

\bibitem[Dac08]{Dac}
B.~Dacorogna.
\newblock {\em {Direct Methods in the Calculus of Variations}}.
\newblock Springer-Verlag New York, 2 edition, 2008.

\bibitem[DP85]{DiPerna}
R.~Di~Perna.
\newblock {Compensated compactness and general systems of conservation laws}.
\newblock {\em Trans. Amer. Math. Soc.}, 292:383--420, 1985.

\bibitem[FM99]{FM}
I.~Fonseca and S.~M\"uller.
\newblock {A-quasiconvexity, lower-semicontinuity and Young measures}.
\newblock {\em SIAM J. Math. Anal.}, 30(6):1355--1390, 1999.

\bibitem[GR19]{RG}
A.~Guerra and B.~Rai\c{t}\u{a}.
\newblock {Quasiconvexity, null Lagrangians, and Hardy space integrability
  under constant rank constraints}.
\newblock {\em https://arxiv.org/abs/1909.03923}, 2019.

\bibitem[Kri99]{Kristensen}
J.~Kristensen.
\newblock On the non-locality of quasiconvexity.
\newblock {\em nnales de l'Institut Henri Poincar\'{e}}, 6:1--13, 1999.

\bibitem[Mey65]{Meyers}
N.~Meyers.
\newblock Quasiconvexity and the lower semicontinuity of multiple variational
  integrals of any order.
\newblock {\em Transactions of the American Mathematical Society},
  119(1):125--149, 1965.

\bibitem[Mor66]{Mor66}
C.~Morrey.
\newblock {\em Multiple Integrals in Calculus of Variations}.
\newblock Springer, 1966.

\bibitem[M{\"u}l99]{Mlecture}
S.~M{\"u}ller.
\newblock Variational models for microstructure and phase transitions.
\newblock In {\em Calculus of {Variations} and {Geometric} {Evolution}
  {Problems}: {Lectures} given at the 2nd {Session} of the {Centro}
  {Internazionale} {Matematico} {Estivo} ({C}.{I}.{M}.{E}.) held in {Cetraro},
  {Italy}, {June} 15-22, 1996}, Lecture {Notes} in {Mathematics}, pages
  85--210. Springer, Berlin, Heidelberg, 1999.

\bibitem[Mur78]{Murat2}
F.~Murat.
\newblock {Compacit\'e par compensation}.
\newblock {\em Ann. Sc. Norm. Super. Pisa, Cl. Sci., IV. Ser}, 5:489--507,
  1978.

\bibitem[Mur81]{Murat}
F.~Murat.
\newblock {Compacit\'e par compensation: condition necessaire et suffisante de
  continuit\'e faible sous une hypoth\'ese de rang constant}.
\newblock {\em Ann. Sc. Norm. Sup. Pisa}, 8:68--102, 1981.

\bibitem[Rai19]{Raita}
B.~Rai\c{t}\u{a}.
\newblock {Potentials for A-quasiconvexity}.
\newblock {\em Calc. Var.}, 58:105, 2019.

\bibitem[Res67]{Resh}
Y.~Reshetnyak.
\newblock On the stability of conformal mappings in multidimensional spaces.
\newblock {\em Sib. Math. J.}, 8:69--85, 1967.

\bibitem[Rin14]{Rindler}
F.~Rindler.
\newblock {Directional oscillations, concentrations, and compensated
  compactness via microlocal compactness forms}.
\newblock {\em Arch. Ration. Mech. Anal.}, 215:1--63, 2014.

\bibitem[Tar79]{Compcomp}
L.~Tartar.
\newblock {Compensated compactness and applications to partial differential
  equations}.
\newblock In {\em {Nonlinear Analysis and Mechanics: Heriot-Watt Symposium}},
  volume~4, pages 136--212. Pitman Res. Notes Math, 1979.

\bibitem[\v{S}92]{Sverak}
V.~\v{S}ver\'{a}k.
\newblock Rank-one convexity does not imply quasiconvexity.
\newblock {\em Proc. Roy. Soc. Edinburgh}, 120:185--189, 1992.

\end{thebibliography}
